\documentclass[11pt]{nyjm}

\title{Typicality of normal numbers with respect to the Cantor series expansion} 

\author{Bill Mance}  
\address{
              Department of Mathematics \\
              The Ohio State University \\
              231 West 18th Avenue \\
              Columbus, OH 43210-1174
} 
\email{mance@math.ohio-state.edu}   
\thanks{I would like to thank Christian Altomare and Vitaly Bergelson for many useful conversations.} 



\newcommand{\e}{\epsilon}
\newcommand{\la}{\lambda}

\newcommand{\qnk}{Q_n^{(k)}}

\newcommand{\formalsum}{\sum_{n=1}^{\infty} \frac {E_n} {q_1 q_2 \ldots q_n}}

\newcommand{\E}[1]{\hbox{E}\left[  #1  \right]}
\newcommand{\var}[1]{\hbox{Var}\left[  #1  \right]}
\newcommand{\Prob}[1]{\hbox{P}\left(  #1  \right)}

\newtheorem{theorem}{Theorem}[section]
\newtheorem{corollary}[theorem]{Corollary}
\newtheorem{lemma}[theorem]{Lemma}
\newtheorem{proposition}[theorem]{Proposition}

\theoremstyle{definition}
\newtheorem{definition}[theorem]{Definition}

\newcommand{\labeq}[1]{\label{eqn:#1}}
\newcommand{\refeq}[1]{(\ref{eqn:#1})}
\newcommand{\labt}[1]{\label{thm:#1}}
\newcommand{\reft}[1]{Theorem~\ref{thm:#1}}
\newcommand{\labl}[1]{\label{lemma:#1}}
\newcommand{\refl}[1]{Lemma~\ref{lemma:#1}}
\newcommand{\labp}[1]{\label{proposition:#1}}
\newcommand{\refp}[1]{Proposition~\ref{proposition:#1}}
\newcommand{\labd}[1]{\label{definition:#1}}
\newcommand{\refd}[1]{Definition~\ref{definition:#1}}
\newcommand{\labc}[1]{\label{coro:#1}}
\newcommand{\refc}[1]{Corollary~\ref{coro:#1}}

\newcommand{\floor}[1]{\lfloor #1 \rfloor} 
\newcommand{\ceil}[1]{\lceil #1 \rceil}

\begin{document} 
 
\begin{abstract}  
Fix a sequence of integers $Q=\{q_n\}_{n=1}^\infty$ such that $q_n$ is greater than or equal to $2$ for all $n$.
In this paper, we improve upon results by J. Galambos and F. Schweiger showing that almost every (in the sense of Lebesgue measure) real number in $[0,1)$ is $Q$-normal with respect to the $Q$-Cantor series expansion for sequences $Q$ that satisfy a certain condition.  We also provide asymptotics describing the number of occurrences of blocks of digits in the $Q$-Cantor series expansion of a typical number.  The notion of strong $Q$-normality, that satisfies a similar typicality result, is introduced.  Both of these notions are equivalent for the $b$-ary expansion, but strong normality is stronger than normality for the Cantor series expansion.  In order to show this, we provide an explicit construction of a sequence $Q$ and a real number that is $Q$-normal, but not strongly $Q$-normal.  We use the results in this paper to show that under a mild condition on the sequence $Q$, a set satisfying a weaker notion of normality, studied by A. R\'enyi in \cite{Renyi}, will be dense in $[0,1)$.
\end{abstract} 
\maketitle
\tableofcontents


\section{Introduction}
\label{sec:1}

\begin{definition}\labd{1.1} Let $b$ and $k$ be positive integers.  A {\it block of length $k$ in base $b$} is an ordered $k$-tuple of integers in $\{0,1,\ldots,b-1\}$.  A {\it block of length $k$} is a block of length $k$ in some base $b$.  A {\it block} is a block of length $k$ in base $b$ for some integers $k$ and $b$.
\end{definition}

\begin{definition}\labd{1.2} Given an integer $b \geq 2$, the {\it $b$-ary expansion} of a real $x$ in $[0,1)$ is the (unique) expansion of the form
$$
x=\sum_{n=1}^{\infty} \frac {E_n} {b^n}=0.E_1 E_2 E_3 \ldots
$$
such that $E_n$ is in $\{0,1,\ldots,b-1\}$ for all $n$ with $E_n \neq b-1$ infinitely often.
\end{definition}

Denote by $N_n^b(B,x)$ the number of times a block $B$ occurs with its starting position no greater than $n$ in the $b$-ary expansion of $x$.

\begin{definition}\labd{1.3} A real number $x$ in $[0,1)$ is {\it normal in base $b$} if for all $k$ and blocks $B$ in base $b$ of length $k$, one has
\begin{equation} \label{eqn:bnormal1}
\lim_{n \rightarrow \infty} \frac {N_n^{b}(B,x)} {n}=b^{-k}.
\end{equation}
A number $x$ is {\it simply normal in base $b$} if \refeq{bnormal1} holds for $k=1$.
\end{definition}

 Borel introduced normal numbers in 1909 and proved that almost all (in the sense of Lebesgue measure) real numbers in $[0,1)$ are normal in all bases.  The best known example of a number that is normal in base $10$ is due to Champernowne \cite{Champernowne}.  The number
$$
H_{10}=0.1 \ 2 \ 3 \ 4 \ 5 \ 6 \ 7 \ 8 \ 9 \ 10  \ 11 \ 12 \ldots ,
$$
formed by concatenating the digits of every natural number written in increasing order in base $10$, is normal in base $10$.  Any $H_b$, formed similarly to $H_{10}$ but in base $b$, is known to be normal in base $b$. Since then, many examples have been given of numbers that are normal in at least one base.  One  can find a more thorough literature review in \cite{DT} and \cite{KuN}.

The $Q$-Cantor series expansion, first studied by Georg Cantor in \cite{Cantor}, is a natural generalization of the $b$-ary expansion.

\begin{definition}\labd{1.4} $Q=\{q_n\}_{n=1}^{\infty}$ is a {\it basic sequence} if each $q_n$ is an integer greater than or equal to $2$.
\end{definition}

\begin{definition}\labd{1.5} Given a basic sequence $Q$, the {\it $Q$-Cantor series expansion} of a real $x$ in $[0,1)$ is the (unique) expansion of the form
\begin{equation} \label{eqn:cseries}
x=\sum_{n=1}^{\infty} \frac {E_n} {q_1 q_2 \ldots q_n}
\end{equation}
such that $E_n$ is in $\{0,1,\ldots,q_n-1\}$ for all $n$
with $E_n \neq q_n-1$ infinitely often.  We abbreviate \refeq{cseries} with the notation $x=0.E_1E_2E_3\ldots$ with respect to  $Q$.
\end{definition}

Clearly, the $b$-ary expansion is a special case of \refeq{cseries} where $q_n=b$ for all $n$.  If one thinks of a $b$-ary expansion as representing an outcome of repeatedly rolling a fair $b$-sided die, then a $Q$-Cantor series expansion may be thought of as representing an outcome of rolling a fair $q_1$ sided die, followed by a fair $q_2$ sided die and so on.  For example, if $q_n=n+1$ for all $n$, then the $Q$-Cantor series expansion of $e-2$ is
$$
e-2=\frac{1} {2}+\frac{1} {2 \cdot 3}+\frac{1} {2 \cdot 3 \cdot 4}+\ldots
$$
If $q_n=10$ for all $n$, then the $Q$-Cantor series expansion for $1/4$ is
$$
\frac {1} {4}=\frac{2} {10}+\frac {5} {10^2}+\frac {0} {10^3}+\frac {0} {10^4}+\ldots
$$

For a given basic sequence $Q$, let $N_n^Q(B,x)$ denote the number of times a block $B$ occurs starting at a position no greater than $n$ in the $Q$-Cantor series expansion of $x$. Additionally, define
$$
Q_n^{(k)}=\sum_{j=1}^n \frac {1} {q_j q_{j+1} \ldots q_{j+k-1}}.
$$

A. R\'enyi \cite{Renyi} defined a real number $x$ to be normal with respect to $Q$ if for all blocks $B$ of length $1$,
\begin{equation}\label{eqn:rnormal}
\lim_{n \rightarrow \infty} \frac {N_n^Q (B,x)} {Q_n^{(1)}}=1.
\end{equation}
If $q_n=b$ for all $n$, then \refeq{rnormal} is equivalent to {\it simple normality in base $b$}, but not equivalent to {\it normality in base $b$}.  Thus, we want to generalize normality in a way that is equivalent to normality in base $b$ when all $q_n=b$.

\begin{definition}\labd{1.7} A real number $x$ is {\it $Q$-normal of order $k$} if for all blocks $B$ of length $k$,
$$
\lim_{n \rightarrow \infty} \frac {N_n^Q (B,x)} {Q_n^{(k)}}=1.
$$
We say that $x$ is {\it $Q$-normal} if it is $Q$-normal of order $k$ for all $k$. A real number $x$ is {\it $Q$-ratio normal of order $k$} if for all  blocks $B$ and $B'$ of length $k$, we have
$$\lim_{n \to \infty} \frac {N_n^Q(B,x)} {N_n^Q(B',x)}=1.
$$
$x$ is {\it $Q$-ratio normal} if it is $Q$-ratio normal of order $k$ for all positive integers $k$.
\end{definition}

We make the following definitions:

\begin{definition}\labd{1.8} A basic sequence $Q$ is {\it $k$-divergent} if
$\lim_{n \rightarrow \infty} Q_n^{(k)}=\infty$.
$Q$ is {\it fully divergent} if $Q$ is $k$-divergent for all $k$.  $Q$ is {\it $k$-convergent} if it is not $k$-divergent.
\end{definition}

\begin{definition}\labd{1.6} A basic sequence $Q$ is {\it infinite in limit} if $q_n \rightarrow \infty$.
\end{definition}

For $Q$ that are infinite in limit,
it has been shown that the set of all $x$ in $[0,1)$ that are $Q$-normal of order $k$ has full Lebesgue measure if and only if $Q$ is $k$-divergent \cite{Renyi}.  Therefore, if $Q$ is infinite in limit, then the set of all $x$ in $[0,1)$ that are $Q$-normal has full Lebesgue measure if and only if $Q$ is fully divergent.  Suppose that $Q$ is $1$-divergent. Given an arbitrary non-negative integer $a$, F. Schweiger \cite{Sch} proved that for almost every $x$ with $\epsilon > 0$, one has
$$
N_n( (a), x ) = Q_n^{(1)} + O\left( \sqrt{Q_n^{(1)}} \cdot  \log^{3/2+\epsilon} Q_n^{(1)}     \right).
$$
J. Galambos proved an even stronger result in \cite{Galambos}.  He showed that for almost every $x$ in $[0,1)$ and for all non-negative integers $a$,
$$
N_n^Q((a),x)=Q_n^{(1)}+O\left(\sqrt{Q_n^{(1)}} \left(\log \log Q_n^{(1)} \right)^{1/2} \right).
$$

We provide the following main results:

\begin{enumerate}
\item A notion of strong $Q$-normality is provided and we construct an explicit example of a basic sequence $Q$ and a real number that is $Q$-normal, but not strongly $Q$-normal (\reft{construction}).
\item (\reft{typicalasymptotics}) If $Q$ is a basic sequence that is infinite in limit and $B$ is a block of length $k$, then for almost every real number $x$ in $[0,1)$, we have
$$
N_n^Q(B,x)=\qnk+O\left(\sqrt{\qnk} \left(\log \log \qnk \right)^{1/2} \right).
$$
\item If $Q$ is infinite in limit, then almost every real number is $Q$-normal of order $k$ if and only if $Q$ is $k$-divergent (\reft{aeQnok}).
\item If $Q$ is $k$-convergent for some $k$, then the set of numbers that are $Q$-normal is empty (\refp{fewnormal2}).  If $Q$ is infinite in limit, then the set of $Q$-ratio normal numbers is dense in $[0,1)$ (\refc{rationormalisdense}).
\end{enumerate}

\section{Strongly Normal Numbers}
\label{sec:2}
\subsection{Basic definitions and results}
In this section, we will introduce a notion of normality that is stronger than $Q$-normality.  This notion of normality will arise naturally later in this paper and will be useful for studying the typicality of $Q$-normal numbers.  We will first need to make definitions similar to those of $N_n^Q(B,x)$ and $\qnk$.

Given a real number $x \in [0,1)$, a basic sequence $Q$, a block $B$ of length $k$, a positive integer $p \in [1,k]$, and a positive integer $n$, we will denote by $N_{n,p}^Q(B,x)$ the number of times that the block $B$ occurs in the $Q$-Cantor series expansion of $x$ with starting position of the form $j \cdot k+p$ for $0 \leq j <  \frac {n} {k}$.

If  $n$ and $k$ are positive integers, define
$$
\rho(n,k)=\ceil{n/k}-1=\max \left\{i \in \mathbb{Z} : i < \frac {n} {k} \right\}.
$$

Suppose that $Q$ is a basic sequence and that $n,p$, and $k$  are positive integers with $p \in [1,k]$.  We will write
$$
Q_{n,p}^{(k)}=\sum_{j=0}^{\rho(n,k)} \frac {1} {q_{jk+p}q_{jk+p+1} \cdots q_{jk+p+k-1}}.
$$

\begin{definition}
Let $k$ be a positive integer.  Then a basic sequence $Q$ is {\it strongly $k$-divergent}
\footnote{It is not true that $k$-divergent basic sequences must be strongly $k$-divergent. The following example of a $2$-divergent basic sequence that is not strongly $2$-divergent was suggested by C. Altomare (verbal communication):  let the basic sequence $Q=\{q_n\}$ be given by
$$
q_n=\left\{ \begin{array}{ll}
\max(2,\floor{n^{1/4}}) & \textrm{if $n \equiv 0 \pmod{4}$}\\
\max(2,\floor{n^{1/4} \cdot \log^2 n}) & \textrm{if $n \equiv 1 \pmod{4}$}\\
\max(2,\floor{n^{3/4}}) & \textrm{if $n \equiv 2 \pmod{4}$}\\
\max(2,\floor{n^{3/4} \cdot \log^2 n}) & \textrm{if $n \equiv 3 \pmod{4}$}
\end{array} \right. .
$$}
 if for all positive integers $p$ with $p \in [1,k]$, we have $\lim_{n \to \infty} Q_{n,p}^{(k)}= \infty$.
A basic sequence $Q$ is {\it strongly fully divergent} if it is strongly $k$-divergent for all $k$.
\end{definition}

Given a real number $x \in [0,1)$, a basic sequence $Q$, a block $B$ of length $k$, a positive integer $p \in [1,k]$, and a positive integer $n$, we will denote by $N_{n,p}^Q(B,x)$ the number of times the block $B$ occurs in the $Q$-Cantor series expansion of $x$ with positions of the form $j \cdot k+p$ for $0 \leq j <  \frac {n} {k}$.

\begin{definition}
Suppose that $Q$ is a basic sequence.  A real number $x$ in $[0,1)$ is {\it strongly $Q$-normal of order $k$} if for all blocks $B$ of length $m \leq k$ and all $p \in [1,m]$, we have
$$
\lim_{n \rightarrow \infty} \frac {N_{n,p}^Q(B,x)} {Q_{n,p}^{(m)}}=1.
$$
A real number $x$ is {\it strongly $Q$-normal}
if it is strongly $Q$ normal of order $k$ for all $k$.
\end{definition}

We will use the following lemmas frequently and without mention:

\begin{lemma}
Given a real number $x \in [0,1)$, a basic sequence $Q$, a block $B$ of length $k$, a positive integer $p \in [1,k]$, and a positive integer $n$, we have
$$
N_{n,1}^Q(B,x)+N_{n,2}^Q(B,x)+\ldots+N_{n,k}^Q(B,x)=N_n^Q(B,x)+O(1) \hbox{ and}
$$
$$
Q_{n,1}^{(k)}+Q_{n,2}^{(k)}+\ldots+Q_{n,k}^{(k)}=\qnk+O(1).
$$
\end{lemma}

\begin{proof}
This follows directly from the definitions of $N_n^Q(B,x)$ and $Q_{n}^{(k)}$.
\end{proof}

\begin{lemma}\labl{sumlittleo}
If $g_1,g_2,\ldots,g_n$ are non-negative functions on the natural numbers, then
$$
o(g_1)+o(g_2)+\ldots+o(g_n)=o(g_1+g_2+\ldots+g_n).
$$
\end{lemma}


\begin{theorem}\labt{SNisN}
If $Q$ is a basic sequence and $x$ is strongly $Q$-normal of order $k$, then $x$ is $Q$-normal of order $k$.
\end{theorem}

\begin{proof}
Let $m \leq k$ be a positive integer and let $B$ be a block of length $k$.  Since $x$ is strongly $Q$-normal of $k$, we know that for all $p \in [1,m]$, $N_{n,p}^Q(B,x)=Q_{n,p}^{(k)}+o\left(Q_{n,p}^{(k)}\right)$. Thus, we see that
$$
N_n^Q(B,x)=\sum_{p=1}^m N_{n,p}^Q(B,x)=\sum_{p=1}^m \left( Q_{n,p}^{(k)}+o\left(Q_{n,p}^{(k)}\right) \right)
$$
$$
=\sum_{p=1}^m Q_{n,p}^{(k)}+o\left( \sum_{p=1}^m Q_{n,p}^{(k)} \right)=\qnk+o\left(\qnk \right),
$$
so $\lim_{n \to \infty} \frac {N_n^Q(B,x)} {\qnk}=1$. Therefore, $x$ is $Q$-normal of order $k$.
\end{proof}

\begin{corollary}\labc{snisn}
Suppose that $Q$ is a basic sequence.  If $x$ is strongly $Q$-normal, then $x$ is $Q$-normal.
\end{corollary}

\subsection{Construction of a number that is $Q$-normal, but not strongly $Q$-normal of order $2$}

In this subsection, we will work towards giving an example of a basic sequence $Q$ and a real number $x$ that is $Q$-normal, but not strongly $Q$-normal of order $2$.  We will use the conventions found in \cite{Mance}.

Given a block $B$, $|B|$ will represent the length of $B$. Given non-negative integers $l_1,l_2,\ldots,l_n$, at least one of which is positive, and blocks $B_1,B_2,\ldots,B_n$, the block $B=l_1B_1 l_2B_2 \ldots l_n B_n$
will be the block of length $l_1 |B_1|+\ldots+l_n |B_n|$ formed by concatenating $l_1$ copies of $B_1$, $l_2$ copies of $B_2$, through $l_n$ copies of $B_n$.  For example, if $B_1=(2,3,5)$ and $B_2=(0,8)$, then $2B_1 1B_20B_2=(2,3,5,2,3,5,0,8)$.  We will need the following definitions:

\begin{definition}\labd{1.9} A {\it weighting} $\mu$ is a collection of functions $\mu^{(1)},\mu^{(2)},\mu^{(3)},\ldots$ with $\sum_{j=0}^{\infty} \mu^{(1)}(j)=1$ such that for all $k$, $\mu^{(k)}:\{0,1,2,\ldots\}^k \rightarrow [0,1]$ and $\mu^{(k)}(b_1,b_2,\ldots,b_k)=\sum_{j=0}^{\infty} \mu^{(k+1)}(b_1,b_2,\ldots,b_k,j)$.
\end{definition}

\begin{definition}\labd{1.10} The {\it uniform weighting in base $b$} is the collection $\lambda_b$ of functions  $\lambda_b^{(1)},\lambda_b^{(2)},\lambda_b^{(3)},\ldots$ such that for all $k$ and blocks $B$ of length $k$ in base $b$
\begin{equation}
\lambda_b^{(k)}(B)=b^{-k}.
\end{equation}
\end{definition}

\begin{definition}\labd{1.11} Let $p$ and $b$ be positive integers such that $1 \leq p \leq b$.  A weighting $\mu$ is {\it $(p,b)$-uniform} if for all $k$ and blocks $B$ of length $k$ in base $p$, we have
\begin{equation}
\mu^{(k)}(B)=\lambda_b^{(k)}(B)=b^{-k}.
\end{equation}
\end{definition}

Given blocks $B$ and $y$, let $N(B,y)$ be the number of
occurrences of the block~$B$ in the block~$y$.

\begin{definition}\labd{1.12} Let $\epsilon$ be a real number such that $0 < \epsilon < 1$ and let $k$ be  a positive integer.
Assume that $\mu$ is a weighting.  A block of digits $y$ is {\it $(\epsilon,k,\mu)$-normal }\footnote{\refd{1.12} is a generalization of the concept of $(\epsilon,k)$-normality, originally due to Besicovitch \cite{Besicovitch}.} if for all blocks $B$ of length $m \leq k$, we have
\begin{equation}
\mu^{(m)}(B)|y|(1-\epsilon) \le N(B,y) \le \mu^{(m)}(B)|y|(1+\epsilon).
\end{equation}
\end{definition}

For the rest of this subsection, we use the following conventions.  Given sequences of non-negative integers $\{l_i\}_{i=1}^\infty$ and $\{b_i\}_{i=1}^\infty$ with each $b_i\ge 2$
and a sequence of blocks $\{x_i\}_{i=1}^\infty$, we set
\begin{equation}
L_i=|l_1 x_1 \ldots l_i x_i|=\sum_{j=1}^i  l_j |x_j|,
\end{equation}
\begin{equation}
q_n=b_i \textrm{\ for $L_{i-1} < n \leq L_i$},
\end{equation}
and
\begin{equation}
Q=\{q_n\}_{n=1}^{\infty}.
\end{equation}
Moreover, if $(E_1,E_2,\ldots)=l_1x_1 l_2 x_2 \ldots$, we set
\begin{equation}
x=\sum_{n=1}^{\infty} \frac {E_n} {q_1 q_2 \ldots q_n}.
\end{equation}

Given $\{q_n\}_{n=1}^{\infty}$ and $\{l_i\}_{i=1}^{\infty},$ it is
assumed that $x$ and $Q$ are given by the formulas above. 

\begin{definition}\labd{1.13} A {\it block friendly family} is a $6$-tuple $W=\{(l_i,b_i,p_i,\epsilon_i,k_i,\mu_i)\}_{i=1}^{\infty}$ with non-decreasing sequences of non-negative integers $\{l_i\}_{i=1}^{\infty}$, $\{b_i\}_{i=1}^{\infty}$, $\{p_i\}_{i=1}^{\infty}$ and $\{k_i\}_{i=1}^{\infty}$, for which $b_i \geq 2$, $b_i \rightarrow \infty$ and $p_i \rightarrow \infty$, such that $\{\mu_i\}_{i=1}^{\infty}$ is a sequence of $(p_i,b_i)$-uniform weightings and $\{\epsilon_i\}_{i=1}^{\infty}$ strictly decreases to $0$.
\end{definition}

\begin{definition}\labd{1.14} Let $W=\{(l_i,b_i,p_i,\epsilon_i,k_i,\mu_i)\}_{i=1}^{\infty}$ be a block friendly family.  If $\lim k_i=K<\infty$, then let $R(W)=\{0,1,2,\ldots,K\}$. Otherwise, let $R(W)=\{0,1,2,\ldots\}$.
A sequence $\{x_i\}_{i=1}^\infty$ of $(\e_i,k_i,\mu_i)$-normal blocks of
non-decreasing length
is said to be {\it $W$-good} if for all $k$ in $R$,
the following three conditions hold:
\begin{equation}\label{eqn:good1}
\frac {b_i^k} {\epsilon_{i-1}-\epsilon_i}=o(|x_i|);
\end{equation}
\begin{equation}\label{eqn:good2}
\frac {l_{i-1}} {l_i} \cdot \frac {|x_{i-1}|} {|x_i|}=o(i^{-1}b_i^{-k});
\end{equation}
\begin{equation}\label{eqn:good3}
\frac {1} {l_i} \cdot \frac {|x_{i+1}|} {|x_i|}=o(b_i^{-k}).
\end{equation}
\end{definition}

We now state a key theorem of \cite{Mance}.

\begin{theorem}\labt{oldmain} Let $W$ be a block friendly family and $\{x_i\}_{i=1}^{\infty}$ a $W$-good sequence. If~$k~\in~R(W)$, then $x$ is $Q$-normal of order $k$.  If $k_i \rightarrow \infty$, then $x$ is $Q$-normal.
\end{theorem}

If $b$ and $w$ are positive integers where $b$ is greater than or equal to $2$ and $w \geq 3$ is odd, then we let $C_{b,w}$ be one of the blocks formed by concatenating all the blocks of length $w$ in base $b$ in such a way that there are at least twice as many copies of the block $(0)$ at odd positions as the block $(1)$.  For example, we could pick
$$
C_{2,3}=1(0,0,0)1(1,0,1)1(0,1,0)1(0,0,1)1(0,1,1)1(1,0,0)1(1,1,0)1(1,1,1)
$$
$$
=(0,0,0,1,0,1,0,1,0,0,0,1,0,1,1,1,0,0,1,1,0,1,1,1),
$$
which has $9$ copies of $(0)$ at the odd positions and $3$ copies of $(1)$ at the odd positions. Note that $|C_{b,w}|=wb^w$.
The next lemma is proven identically to Lemma 4.2 in \cite{Mance}:

\begin{lemma}\labl{l4.2}
If $K<w$ and $\epsilon=\frac {K} {w}$, then $C_{b,w}$ is $(\epsilon,K,\lambda_b)$-normal.
\end{lemma}


\begin{theorem}\labt{construction}\footnote{\reft{oldmain} may be used to construct other explicit examples of $Q$-normal numbers that satisfy some unusual conditions.  Given a basic sequence $Q$, we say that $x$ is {\it $Q$-distribution normal} if the sequence $\{q_1q_2 \cdots q_n x\}_n$ is uniformly distributed mod $1$.  \cite{AlMa} uses \reft{oldmain} to give an example of a basic sequence $Q$ and a real number $x$ such that $x$ is $Q$-normal, but $q_1q_2 \cdots q_n x  \pmod{1} \to 0$, so $x$ is not $Q$-distribution normal.}
There exists a basic sequence $Q$ and a real number $x$ such that $x$ is $Q$-normal, but not strongly $Q$-normal of order $2$.
\end{theorem}

\begin{proof}
Let $x_1=(0,1)$, $b_1=2$, and $l_1=0$.  For $i \geq 2$, let $x_i=C_{2i,(2i+1)^2}$, $b_i=2i$, and $l_i=(2i)^{9i+8}$.  Set $\epsilon_1=1/2$, $k_1=1$, $p_1=2$ and $\mu_1=\lambda_2$.  For $i \geq 2$, put $\epsilon_i=1/(2i+1)$, $k_i=2i+1$, $p_i=b_i$, $\mu_i=\lambda_{2i}$, and $W=\{(l_i,b_i,p_i,\epsilon_i,k_i,\mu_i)\}_{i=1}^{\infty}$.  Thus, since $x_i=C_{b,w}$ where $b=2i$ and $w=(2i+1)^2$, $x_i$ is $(\epsilon_i,k_i,\lambda_{b_i})$-normal  by \refl{l4.2}.

In order to show that $\{x_i\}$ is a $W$-good sequence we need to verify \refeq{good1}, \refeq{good2}, and \refeq{good3}.
Since $k_i \rightarrow \infty$, we let $k$ be an arbitrary positive integer. We will make repeated use of the fact that $|x_i|=(2i+1)^2 \cdot (2i)^{(2i+1)^2}$.
We first verify $(\ref{eqn:good1})$:
$$
\lim_{i \rightarrow \infty} |x_i| \Bigg/ \left( \frac {(2i)^k} {\frac {1} {2(i-1)+1}-\frac {1} {2i+1}} \right)
=\lim_{i \rightarrow \infty} \frac {2(2i+1)^2 \cdot (2i)^{(2i+1)^2}} {(2i)^k \cdot (4i^2-1)}=\infty.
$$
We next verify $(\ref{eqn:good2})$.  Since $l_{i-1}/l_i<1$, $(2i-1)^2/(2i+1)^2<1$ and 
$$\left(1-\frac {1} {i} \right)^{(2i+1)^2}<e^{-2(2i+1)},$$
we have
$$
\lim_{i \rightarrow \infty} \frac {\frac {l_{i-1}} {l_i} \cdot \frac {x_{i-1}} {x_i}} {i^{-1} (2i)^{-k}}
\leq \lim_{i \rightarrow \infty} i \cdot (2i)^{k} \cdot 1 \cdot \frac {(2i-1)^2} {(2i+1)^2} \cdot \frac {(2i-2)^{(2i-1)^2}} {(2i)^{(2i+1)^2}}
$$
$$
\leq \lim_{i \rightarrow \infty} i(2i)^{k} \cdot 1 \cdot \left(1-\frac {1} {i} \right)^{(2i+1)^2} \cdot (2i-2)^{-8i}
 \leq \lim_{i \rightarrow \infty} i(2i+1)^{k} e^{-2(2i+1)} (2i-2)^{-8i}=0.
$$
Lastly, we will verify $(\ref{eqn:good3})$.  Since $(2i+3)^2/(2i+1)^2 \leq 2$, $(1+2/(2i+1))^{8i}<e^8$, and 
$$\left(1+\frac {2} {2i+1} \right)^{(2i+1)^2}<2e^{2(2i+1)},$$
we have
$$
\lim_{i \rightarrow \infty} \frac {\frac {1} {l_i} \cdot \frac {|x_{i+1}|} {|x_i|}} {(2i)^{-k}} 
= \lim_{i \rightarrow \infty} (2i)^{-9i-8+k} \cdot \frac {(2i+3)^2} {(2i+1)^2} \cdot \frac {(2i+2)^{(2i+3)^2}} {(2i)^{(2i+1)^2}}
$$
$$
\leq \lim_{i \rightarrow \infty} (2i)^{-9i-8+k} \cdot 2 \cdot \left(1+\frac {1} {i} \right)^{(2i+1)^2} \cdot (2i+2)^{(8i+8)} 
$$
$$
\leq \lim_{i \rightarrow \infty} 4e^{2(2i+1)} \left(1+\frac {1} {i}\right)^{8i+8} (2i)^{-i+k}
 \leq \lim_{i \rightarrow \infty} 4e^{2(2i+1)+8} \cdot (2i)^{-i+k}=0.
$$
Since $\lambda_{b_i}$ is $(p_i,b_i)$-uniform, $\{x_i\}$ is a $W$-good sequence and by \reft{oldmain}, $x$ is $Q$-normal. 

Since the length of each block $x_i$ is even, so there will always be at least twice as many copies of the block $(0)$ as the block $(1)$ in any initial segment of digits of $x$, so $x$ is not strongly $Q$-normal of order $2$.
\end{proof}

\section{Random Variables Associated With Normality}
\label{sec:3}

For this section, we must recall a few basic notions from probability theory.  Given a random variable $X$, we will denote the expected value of $X$ as $\E{X}$.  We will denote the variance of $X$ as $\var{X}$.  Lastly, $\Prob{X=j}$ will represent the probability that $X=j$.

We consider $x$ as a random variable which has uniform distribution on the interval $[0,1)$.  If $x=0.E_1(x)E_2(x)E_3(x)\ldots$ with respect to $Q$, then we consider $E_1(x), E_2(x),E_3(x),\ldots$ to be random variables.  So for all $n$, we have
$$
\Prob{E_n(x)=j}=\left\{ \begin{array}{ll}
\frac {1} {q_n} & \textrm{if $0 \leq j \leq q_n-1$}\\
0		& \textrm{if $j \geq q_n$}
\end{array} \right. .
$$

\begin{lemma}\labl{Eareindependent}
If $Q$ is a basic sequence, then the random variables 
$E_1(x), E_2(x),E_3(x),\ldots$
are independent.
\end{lemma}

\begin{proof}
Suppose that $n_1$ and $n_2$ are distinct positive integers and that $0 \leq F_j < q_j-1$ for all $j$.  Then
$$
\Prob{E_{n_1}(x)=F_{n_1}, E_{n_2}(x)=F_{n_2}}=\la \left(  \{ x \in [0,1) : E_{n_1}(x)=F_{n_1} \hbox{ and } E_{n_2}(x)=F_{n_2} \} \right)
$$
$$
= \frac {1} {q_{n_1}q_{n_2}}=\frac {1} {q_{n_1}} \cdot \frac {1} {q_{n_2}}=\Prob{E_{n_1}(x)=F_{n_1}} \cdot \Prob{E_{n_2}(x)=F_{n_2}}.
$$
\end{proof}

Suppose that $Q$ is a basic sequence, $b$ is a natural number, $B$ is a block of length $k$, and $m=ik+p$ is an integer with $p \in [0,k-1]$. We set
$$
\zeta^Q_{b,n}(x)=\left\{ \begin{array}{ll}
1 & \textrm{if $E_{n}(x)=b$}\\
0 & \textrm{if $E_{n}(x) \neq n$}
\end{array} \right. ,
$$
$$
\zeta^Q_{B,i,p}(x)=\left\{ \begin{array}{ll}
1 & \textrm{if $E_{ik+p,k}(x)=B$}\\
0 & \textrm{if $E_{ik+p,k}(x) \neq B$}
\end{array} \right. ,
$$
$$
F_m^{(k)}=\E{\zeta^Q_{B,i,p}(x)}, V_m^{(k)}=\var{\zeta^Q_{B,i,p}(x)} \hbox{, and } t_{n,p}^{(k)}=\sum_{i=0}^{\rho(n,k)} V_{ik+p}^{(k)}.
$$

\begin{lemma}\labl{rnindependent}
For all non-negative integers $b$, the random variables $\zeta_{b,1}^Q(x), \zeta_{b,2}^Q(x), \zeta_{b,3}^Q(x), \ldots$ are independent.  
\end{lemma}

\begin{proof}
This follows directly from \refl{Eareindependent} as the random variables $E_1(x), E_2(x), E_3(x), \ldots$ are independent.
\end{proof}

\begin{lemma}\labl{productrip}
If $B=(b_1,b_2,\ldots,b_k)$ is a block of length $k$, then
$$
\zeta^Q_{B,i,p}(x)=\zeta^Q_{b_1,ik+p}(x) \cdot \zeta^Q_{b_2,ik+p+1}(x) \cdots \zeta^Q_{b_k,ik+p+k-1}(x).
$$
\end{lemma}

\begin{proof}
By definition,
$$
\zeta^Q_{B,i,p}(x)=\left\{ \begin{array}{ll}
1 & \textrm{if $E_{ik+p,k}=B$}\\
0 & \textrm{if $E_{ik+p,k} \neq B$}
\end{array} \right. ,
$$
or in other words, $\zeta^Q_{B,i,p}(x)=1$ if
$$
\zeta^Q_{b_1,ik+p}(x)= \zeta^Q_{b_2,ik+p+1}(x)=\ldots=\zeta^Q_{b_k,ik+p+k-1}(x)=1
$$
and $\zeta^Q_{B,i,p}(x)=0$ otherwise.
\end{proof}

\begin{corollary}\labc{rareindep}
For all blocks $B=(b_1,b_2,\ldots,b_k)$ of length $k$ and non-negative integers $p_1, p_2 \in [1,k]$, $i_1$, and $i_2$ with $(i_1,p_1) \neq (i_2,p_2)$, the random variables $\zeta^Q_{B,i_1,p_1}(x)$ and $\zeta^Q_{B,i_2,p_2}(x)$ are independent.
\end{corollary}

\begin{proof}
Using \refl{rnindependent} and \refl{productrip}, we see that 
$$
\E{\zeta^Q_{B,i_1,p_1}(x) \cdot \zeta^Q_{B,i_2,p_2}(x)}
=\E{\left(\Pi_{j=0}^{k-1} \zeta^Q_{b_j,i_1 k+p_1+j}(x) \right) \cdot \left(\Pi_{j=0}^{k-1} \zeta^Q_{b_j,i_2 k+p_2+j}(x) \right)}
$$
$$
=\left(\Pi_{j=0}^{k-1} \E{ \zeta^Q_{b_j,i_1 k+p_1+j}(x) } \right) \cdot \left(\Pi_{j=0}^{k-1} \E{ \zeta^Q_{b_j,i_2 k+p_2+j}(x) }\right)
$$
$$
=\E{ \Pi_{j=0}^{k-1}  \zeta^Q_{b_j,i_1 k+p_1+j}(x) } \cdot \E{\Pi_{j=0}^{k-1} \zeta^Q_{b_j,i_2 k+p_2+j}(x) }=\E{ \zeta^Q_{B,i_1,p_1}(x) } \cdot \E{\zeta^Q_{B,i_2,p_2}(x)}.
$$
\end{proof}

\begin{lemma}\labl{compofexpected}
If $B=(b_1,b_2,\ldots,b_k)$ is a block of length $k$, then
$$
F_m^{(k)}=\frac {1} {q_{ik+p} q_{ik+p+1} \ldots q_{ik+p+k-1}} \hbox{ and}
$$
$$
V_m^{(k)}=\frac {1} {q_{ik+p} q_{ik+p+1} \ldots q_{ik+p+k-1}}-\left( \frac {1} {q_{ik+p} q_{ik+p+1} \ldots q_{ik+p+k-1}} \right)^2.
$$
\end{lemma}

\begin{proof}
We first compute the expected value of $\zeta^Q_{B,i,p}(x)$. By \refl{rnindependent} and \refl{productrip}, we see that
$$
\E{\zeta^Q_{B,i,p}(x)}=\E{\zeta^Q_{b_1,ik+p}(x) \cdot \zeta^Q_{b_2,ik+p+1}(x) \cdots \zeta^Q_{b_k,ik+p+k-1}(x)}
$$
$$
=\E{\zeta^Q_{b_1,ik+p}(x)} \cdot \E{\zeta^Q_{b_2,ik+p+1}(x)} \cdots \E{\zeta^Q_{b_k,ik+p+k-1}(x)}
$$
$$
=\frac {1} {q_{ik+p}} \cdot \frac {1} {q_{ik+p+1}} \cdots \frac {1} {q_{ik+p+k-1}} = \frac {1} {q_{ik+p} q_{ik+p+1} \cdots q_{ik+p+k-1}}.
$$
Next, we recall that
$\var{\zeta^Q_{B,i,p}(x)}=\E{\zeta^Q_{B,i,p}(x)^2}-\E{\zeta^Q_{B,i,p}(x)}^2$.
Since $\zeta^Q_{B,i,p}(x)$ may only be $0$ or $1$, we see that 
$\left( \zeta^Q_{B,i,p}(x) \right)^2=\zeta^Q_{B,i,p}(x)$, so
$$
\var{\zeta^Q_{B,i,p}(x)}=\frac {1} {q_{ik+p} q_{ik+p+1} \cdots q_{ik+p+k-1}}-\left( \frac {1} {q_{ik+p} q_{ik+p+1} \cdots q_{ik+p+k-1}} \right)^2.
$$
\end{proof}

Lastly, we remark that $Q_{n,p}^{(k)}=\sum_{i=0}^{\rho(n,k)} F_{ik+p}^{(k)}$ by \refl{compofexpected} and will use this fact frequently and without mention.

\section{Typicality of Normal Numbers}
\label{sec:4}

We will need the following: 

\begin{theorem}\labt{iteratedlog1}\footnote{See, for example, \cite{Vervaat}}
Let $X_1, X_2, \ldots, X_n$ be independent random variables.  Assume that there exists a constant $c>0$ such that $|X_j|<c$ for all $j$.  Let $G_j=\E{X_j}, U_j=\var{X_j}$, and $t_n=\sum_{j=1}^n U_j$.
If $t_n \rightarrow \infty$, then, with probability one,
$$
\limsup_{n \rightarrow \infty} \frac {X_1+X_2+\ldots+X_n-G_1-G_2-\ldots G_n} {\sqrt{2t_n \log \log t_n}}=1.
$$
\end{theorem}

\begin{corollary}\labc{iteratedlog2}
Under the same assumptions of \reft{iteratedlog1}, with probability one,
$$
X_1+X_2+\ldots+X_n=G_1+G_2+\ldots+G_n+O \left( t_n^{1/2} (\log \log t_n)^{1/2} \right).
$$
\end{corollary}

We will also need the Borel-Cantelli Lemma:

\begin{theorem}\labt{BorelCantelli}(The Borel Cantelli Lemma)
If $\sum_{n=1}^{\infty} \Prob{A_n} < \infty$, then $\Prob{A_n \hbox{ i.o.}}=0.$
\end{theorem}

Given a basic sequence $Q$, we will define $t_{n,p}^{(k)}=\sum_{i=0}^{\rho(n,k)} V_{jk+p}^{(k)}$.

\begin{lemma}\labl{tnp}
If $Q$ is a basic sequence and  $n,k$, and $p$ are positive integers  with $p \in [1,k]$, then
$$
\frac {1} {2} Q_{n,p}^{(k)} \leq t_{n,p}^{(k)} < Q_{n,p}^{(k)}.
$$
\end{lemma}
\begin{proof}
$$
t_{n,p}^{(k)}
=\sum_{i=0}^{\rho(n,k)} \left( \frac {1} {q_{ik+p} q_{ik+p+1} \cdots q_{ik+p+k-1}}-\left( \frac {1} {q_{ik+p} q_{ik+p+1} \cdots q_{ik+p+k-1}} \right)^2 \right)
$$
$$
< \sum_{i=0}^{\rho(n,k)} \frac {1} {q_{ik+p} q_{ik+p+1} \cdots q_{ik+p+k-1}}=\sum_{i=0}^{\rho(n,k)} F_{ik+p}^{(k)}=Q_{n,p}^{(k)}.
$$
To show the other direction of the inequality, we recall that since $Q$ is a basic sequence, $q_m \geq 2$ for all $m$, so for all $i$
$$
\sum_{i=0}^{\rho(n,k)} \left( \frac {1} {q_{ik+p} q_{ik+p+1} \cdots q_{ik+p+k-1}}-\left( \frac {1} {q_{ik+p} q_{ik+p+1} \cdots q_{ik+p+k-1}} \right)^2 \right)
$$
$$
\geq \sum_{i=0}^{\rho(n,k)} \left( \frac {1} {q_{ik+p} q_{ik+p+1} \cdots q_{ik+p+k-1}}-\frac {1} {2^k} \left( \frac {1} {q_{ik+p} q_{ik+p+1} \cdots q_{ik+p+k-1}} \right) \right)
$$
$$
\geq \sum_{i=0}^{\rho(n,k)} \frac {1} {2} \cdot \frac {1} {q_{ik+p} q_{ik+p+1} \cdots q_{ik+p+k-1}}=\frac{1} {2} Q_{n,p}^{(k)}.
$$
\end{proof}

\begin{lemma}\labl{lemmastronglynormal}
If $Q$ is infinite in limit and $B$ is a block of length $k$, then for almost every real number $x$ in $[0,1)$, we have
\begin{equation}\labeq{lemmastronglynormal1}
N_{n,p}^Q (B,x)=Q_{n,p}^{(k)}+O\left(\sqrt{Q_{n,p}^{(k)}} \left(\log \log Q_{n,p}^{(k)} \right)^{1/2} \right).
\end{equation}
\end{lemma}

\begin{proof}
We consider two cases.  The first case is when $\lim_{n \to \infty} Q_{n,p}^{(k)}<\infty$. We see that
$$
\lim_{n \to \infty} Q_{n,p}^{(k)}=\lim_{n \to \infty} \sum_{i=0}^{\rho(n,k)} \Prob{\zeta_{B,i,p}^Q=1}< \infty,
$$
so by \reft{BorelCantelli}, we have $\Prob{\zeta_{B,i,p}^Q=1 \hbox{ i.o. }}=0$.  Thus, for almost every $x \in [0,1)$, $\lim_{n \to \infty} N_{n,p}^Q(B,x) < \infty$ and \refeq{lemmastronglynormal1} holds.

Second, we consider the case where $\lim_{n \to \infty} Q_{n,p}^{(k)}=\infty$. By \refl{tnp}, we have
$\lim_{n \to \infty} t_{n,p}^{(k)} \geq \lim_{n \to \infty} Q_{n,p}^{(k)}=\infty$. Note that
$$
N_{n,p}^Q(B,x)=\sum_{i=0}^{\rho(n,k)} \zeta_{B,i,p}(x).
$$
By \refc{iteratedlog2},
$$
N_{n,p}^Q(B,x)=\sum_{i=0}^{\rho(n,k)} F_{ik+p}^{(k)}+O\left(\sqrt{t_{n,p}^{(k)}} \left( \log \log t_{n,p}^{(k)} \right)^{1/2} \right)
$$
for almost every $x \in [0,1)$.  By \refl{tnp}, $t_{n,p}^{(k)} < Q_{n,p}^{(k)}$, so the lemma follows.
\end{proof}

\refl{lemmastronglynormal} allows us to prove the following results on strongly normal numbers:

\begin{theorem}\labt{aeQsnok}
Suppose that $Q$ is strongly $k$-divergent and infinite in limit.  Then almost every $x \in [0,1)$ is strongly $Q$-normal of order $k$.
\end{theorem}

\begin{proof}
Let $B$ be a block of length $m \leq k$ and $p \in [1,m]$.  Then by \refl{lemmastronglynormal}, for almost every $x \in [0,1)$, we have that
$$
N_{n,p}^Q(B,x)=Q_{n,p}^{(m)}+O\left(\sqrt{Q_{n,p}^{(m)}} \left( \log \log Q_{n,p}^{(m)} \right)^{1/2} \right) \hbox{, so}
$$
$$
\frac {N_{n,p}^Q(B,x)} {Q_{n,p}^{(m)}}=1+O\left(\frac {\sqrt{Q_{n,p}^{(m)}} \left( \log \log Q_{n,p}^{(m)} \right)^{1/2}} {Q_{n,p}^{(m)}} \right).
$$
However, $Q$ is strongly $k$-divergent, so $Q_{n,p}^{(m)} \to \infty$ and
$$
\lim_{n \to \infty} \frac {N_{n,p}^Q(B,x)} {Q_{n,p}^{(m)}}=\lim_{n \to \infty} \left( 1+O\left(\frac {\sqrt{Q_{n,p}^{(m)}} \left( \log \log Q_{n,p}^{(m)} \right)^{1/2}} {Q_{n,p}^{(m)}} \right) \right)=1.
$$
Since there are finitely many choices of $m$ and $p$ and only countably many choices of $B$, the result follows.
\end{proof}

\begin{corollary}\labc{aeQsn}
If $Q$ is strongly fully divergent and infinite in limit, then almost every real $x \in [0,1)$ is strongly $Q$-normal.
\end{corollary}

We now work towards proving a result much stronger than \refc{aeQsn} on the typicality of $Q$-normal numbers.  We will need the following lemma in addition to \refl{lemmastronglynormal}:

\begin{lemma}\labl{bigOineq}
 If $Q$ is a basic sequence and $k$ and $p$ are positive integers with $p \in [1,k]$, then
$$
\sum_{p=1}^k \left( Q_{n,p}^{(k)}+O\left(\sqrt{Q_{n,p}^{(k)}} \left(\log \log Q_{n,p}^{(k)} \right)^{1/2} \right) \right) = \qnk + O\left(\sqrt{\qnk} \left(\log \log \qnk \right)^{1/2} \right).
$$
\end{lemma}

\begin{proof}
We first note that
$\sum_{p=1}^k Q_{n,p}^{(k)} \leq Q_n^{(k)}+\left( Q_n^{(k)}-Q_{n-k}^{(k)} \right)$.
Since $ Q_n^{(k)}-Q_{n-k}^{(k)} \leq (k+1)2^{-k} \to 0$, we see that
\begin{equation}\labeq{iso1}
\sum_{p=1}^k Q_{n,p}^{(k)} = Q_n^{(k)}+o(1).
\end{equation}
Next, note that
\begin{equation}\labeq{nastyineq}
\sum_{p=1}^k \sqrt{Q_{n,p}^{(k)}} \left(\log \log Q_{n,p}^{(k)} \right)^{1/2} \leq k \sqrt{\sum_{p=1}^k Q_{n,p}^{(k)}} \left( \log \log \sum_{p=1}^k Q_{n,p}^{(k)} \right)^{1/2}.
\end{equation}
By \refeq{iso1} and \refeq{nastyineq},
\begin{equation}\labeq{firstbigo}
\sum_{p=1}^k O\left(\sqrt{Q_{n,p}^{(k)}} \left(\log \log Q_{n,p}^{(k)} \right)^{1/2} \right)=O\left(\sqrt{\qnk} \left(\log \log \qnk \right)^{1/2} \right).
\end{equation}
Thus, the lemma follows by combining \refeq{iso1} and \refeq{firstbigo}.
\end{proof}

\begin{theorem}\labt{typicalasymptotics}
If $Q$ is a basic sequence that is infinite in limit and $B$ is a block of length $k$, then for almost every real number $x$ in $[0,1)$, we have
$$
N_n^Q(B,x)=\qnk+O\left(\sqrt{\qnk} \left(\log \log \qnk \right)^{1/2} \right).
$$
\end{theorem}

\begin{proof}
We first note that
\begin{equation}\labeq{bong}
 N_n^Q(B,x)=\sum_{p=1}^k N_{n,p}(B,x)+O(1).
\end{equation}
Thus, by \refeq{bong} and \refl{lemmastronglynormal}, for almost every $x \in [0,1)$, we have
\begin{equation}\labeq{bong2}
 N_n^Q(B,x)=\sum_{p=1}^k \left( Q_{n,p}^{(k)}+O\left(\sqrt{Q_{n,p}^{(k)}} \left(\log \log Q_{n,p}^{(k)} \right)^{1/2} \right) \right) +O(1).
\end{equation}
Thus, the theorem follows by applying \refl{bigOineq} to \refeq{bong2}.
\end{proof}

We recall the following standard result on infinite products:

\begin{lemma}\labl{infiniteproduct}
If $\{a_n\}_{n=1}^{\infty}$ is a sequence of real numbers such that $0 \leq a_n <1$ for all $n$,  then the infinite product $\prod_{n=1}^{\infty} (1-a_n)$ converges if and only if the sum $\sum_{n=1}^{\infty} a_n$ is convergent.
\end{lemma}

\begin{theorem}\labt{aeQnok}
Suppose that $Q$ is a basic sequence that is infinite in limit.  Then almost every real number in $[0,1)$ is $Q$-normal of order $k$ if and only if $Q$ is $k$-divergent.
\end{theorem}

\begin{proof}
First, we suppose that $Q$ is $k$-divergent.  Then by \reft{typicalasymptotics}, for almost every $x \in [0,1)$, we have
$$
\lim_{n \to \infty} \frac {N_n^Q(B,x)} {\qnk}=\lim_{n \to \infty} \frac {\qnk+O\left(\sqrt{\qnk} \left(\log \log \qnk \right)^{1/2} \right)} {\qnk}=1.
$$
We now suppose that $Q$ is $k$-convergent.We will now use similar reasoning to that found in \cite{Renyi}.  Set
$B=(0,0,\ldots,0)$ ($k$ zeros).
We will show that the set of real numbers in $[0,1)$ whose $Q$-Cantor series expansion does not contain the block $B$ has positive measure.  Call this set $V$.  We see that
$$
\la(V)=\prod_{n=1}^{\infty} \left(1-\frac {1} {q_nq_{n+1}\cdots q_{n+k-1}} \right).
$$
Set $a_n=q_nq_{n+1}\cdots q_{n+k-1}$.
Since $Q$ is $k$-convergent, we have $\sum a_n < \infty$.  Thus, $\la(V)>0$ by \refl{infiniteproduct}.
\end{proof}

\begin{corollary}\labc{aeQn}
Suppose that $Q$ is a basic sequence that is infinite in limit.  Then almost every real number in $[0,1)$ is $Q$-normal  if and only if $Q$ is fully divergent.
\end{corollary}

\section{Ratio normal numbers}

We are now in a position to compare the prevelance of $Q$-normal numbers to $Q$-ratio normal numbers, depending on properties of the basic sequence $Q$.  In particular, we will show that if $Q$ is infinite in limit, then the set of $Q$-ratio normal numbers is dense in $[0,1)$ even though the set of $Q$-normal numbers may be empty. Suppose that $Q$ is a $k$-convergent basic sequence and define
\begin{equation}
Q_{\infty}^{(k)}=\lim_{n \rightarrow \infty} \qnk<\infty.
\end{equation}

\begin{proposition}\labp{fewnormal2}
If $Q$ is a basic sequence that is $k$-convergent for some $k$, then the set of $Q$-normal numbers is empty.
\end{proposition}

\begin{proof} We make the observation that since $q_n \geq 2$ for all $n$,
$Q_{\infty}^{(k)} \leq \frac {1} {2} Q_{\infty}^{(k-1)}$
for all $k$.  Thus, there exists a $K>0$ such that for all $k>K$, we have
$Q_{\infty}^{(k)} < 1$. Thus, no blocks of length $k>K$ can occur in any $Q$-normal number and the set of $Q$-normal numbers is empty.
\end{proof}

If $B=(b_1,b_2,\cdots,b_k)$ is a block of length $k$, we write $$\max(B)=\max(b_1,b_2,\cdots,b_k).$$
If $E=(E_1,E_2,\cdots)$, then set $E_{n,k}=(E_n,E_{n+1},\cdots,E_{n+k-1})$.

\begin{proposition}\labp{existsrationormal}
If $Q=\{q_n\}_{n=1}^{\infty}$ is  infinite in limit, then there exists a real number that is $Q$-ratio normal.
\end{proposition}

\begin{proof}
Let $Q'=\{q_n'\}_{n=1}^{\infty}$ be any fully divergent basic sequence that is infinite in limit.  Then we know that there exists a $Q'$-normal number by \refc{aeQn}.  Let $x=0.E_1'E_2'E_3'\ldots$ with respect to  $Q'$ be $Q'$-normal and let $E'=(E_1',E_2',\ldots)$.
Set $M_k=\min \{m : q_n > k \ \forall n \geq m \}$, $E_n=\min(E_n',q_n-1)$, and 
$E=(E_1,E_2,\ldots)$.  Suppose that $B$ and $B'$ are two blocks of length $k$ and let
$l=\max(\max(B),\max(B'))+2$.

Thus, if $n>M_l$, then $E'_{n,k}=B$ is equivalent to $E_{n,k}=B$ and $E'_{n,k}=B'$ is equivalent to $E_{n,k}=B'$.  Since $x$ is $Q'$-normal, there are infinitely many occurences of every block.  Additionally, $E_n \leq q_n-1$ for all $n$,  so $\formalsum$ is $Q$-ratio normal.
\end{proof}

\begin{corollary}\labc{rationormalisdense}
If $Q$ is infinite in limit, then the set of numbers that are $Q$-ratio normal is dense in $[0,1)$.
\end{corollary}


\begin{thebibliography}{HD}

\normalsize
\baselineskip=17pt

\bibitem{AlMa} Altomare, C., Mance, B.: Cantor Series Constructions Contrasting Two Notions of Normality. Monatsh. Math. {\bf 164}, 1--22 (2011)

\bibitem{Besicovitch} Besicovitch, A. S. :  The asymptotic distribution of the numerals in the decimal representation of the squares of the natural numbers. Math. Zeit. {\bf 39}, 146--156 (1934)

\bibitem{Champernowne} Champernowne, D. G. :  The construction of decimals normal in the scale of ten. Journal of the London Mathematical Society {\bf 8}, 254--260 (1933)

\bibitem{DT} Drmota, M., Tichy, R. F. : Sequences, Discrepancies and Applications. Springer-Verlag, Berlin Heidelberg (1997)

\bibitem{KuN} Kuipers, L., Niederreiter, H. :  Uniform Distribution of Sequences. Dover, Mineola, NY (2006)

\bibitem{Mance} Mance, B. :  Construction of normal numbers with respect to the $Q$-Cantor series representation for certain $Q$. 
Acta Arith. {\bf 148}, 135--152 (2011)

\bibitem{Renyi} R\'enyi, A. :  On the distribution of the digits in Cantor's series. Mat. Lapok {\bf 7}, 77--100 (1956)

\bibitem{Sch} Schweiger, F. : \"{U}ber den Satz von Borel-R\'enyi in der Theorie der Cantorschen Reihen. Monats. Math. {\bf 74}, 150--153 (1969)

\bibitem{Cantor} {\sc G. Cantor}, {\em \"Uber die einfachen Zahlensysteme}, Zeitschrift f\"ur Math. und Physik {\bf  14}, pp. 121--128  (1869)

\bibitem{Galambos} {\sc J. Galambos}, {\em Representations of real numbers by infinite series}, Lecture Notes in Math. {\bf 502}, Springer-Verlag, Berlin, Hiedelberg, New York, 1976.


\bibitem{Vervaat} Vervaat, W. : Success epochs in Bernoulli trials with applications in number theory.  Math. Centre Tracts, Amsterdam, 1972. Vol 42

\end{thebibliography}
\end{document}